\documentclass[a4paper, 11pt]{article}

\usepackage{mathtools, amssymb, amsthm} 
\usepackage[top=30truemm,bottom=30truemm,left=30truemm,right=30truemm]{geometry}
\usepackage{bm}
\usepackage{enumitem}
\usepackage{xcolor}

\usepackage[
	unicode=true,
	bookmarks=true,
	bookmarkstype=toc,
	bookmarksnumbered=true,
	hidelinks 
	]{hyperref}

\usepackage{tikz}
\usetikzlibrary{intersections, calc, arrows.meta}
\usepackage{subcaption}

\DeclareMathOperator{\AC}{AC}
\DeclareMathOperator{\dom}{dom}
\DeclareMathOperator{\lip}{lip}

\DeclareMathOperator{\Var}{Var}

\newcommand{\dd}{\mathrm{d}}

\newcommand{\ep}{\varepsilon}

\newcommand{\argdot}{\mspace{1.5mu} \cdot \mspace{1.5mu}}

\DeclarePairedDelimiter{\abs}{\lvert}{\rvert}
\DeclarePairedDelimiter{\norm*}{\lVert}{\rVert}

\DeclarePairedDelimiter{\rbra}{\lparen}{\rparen}
\DeclarePairedDelimiter{\sbra}{\lbrack}{\rbrack}

\DeclarePairedDelimiterX{\Set}[2]{\lbrace}{\rbrace}{\mspace{1mu} #1 : #2 \mspace{1mu}}

\theoremstyle{plain}
\newtheorem{theorem}{Theorem}[section]

\newtheorem{lemma}[theorem]{Lemma}
\newtheorem{proposition}[theorem]{Proposition}

\theoremstyle{definition}
\newtheorem{definition}[theorem]{Definition}

\newtheorem{notation}{Notation}

\theoremstyle{remark}
\newtheorem{remark}[theorem]{Remark}

\numberwithin{equation}{section}

\begin{document}

\title{On regularity of mild solutions for autonomous linear retarded functional differential equations}
\author{Junya Nishiguchi\thanks{Mathematical Science Group, Advanced Institute for Materials Research (AIMR), Tohoku University, Katahira 2-1-1, Aoba-ku, Sendai, 980-8577, Japan}
\footnote{E-mail: junya.nishiguchi.b1@tohoku.ac.jp, ORCID: \url{https://orcid.org/0000-0002-0326-2845}}}
\date{}

\maketitle

\begin{abstract}
The notion of mild solutions for autonomous linear retarded functional differential equations (RFDEs) has been introduced in [J. Nishiguchi, Electron.\ J. Qual.\ Theory Differ.\ Equ.\ \textbf{2023}, No.~32, 1--77] for the purpose of defining fundamental matrix solutions and obtaining a variation of constants formula for the RFDEs.
This notion gives a straightforward definition of solutions to the RFDEs under discontinuous history functions compared with previous studies in the literature.
For a given autonomous linear RFDE, it holds that the fundamental matrix solutions are locally Lipschitz continuous on the interval $[0, \infty)$.
However, it is not apparent whether a similar property is true for the mild solutions.
Here we obtain a result which shows the regularity of mild solutions on $[0, \infty)$ for autonomous linear RFDEs.
The result makes clear a connection between the mild solutions and solution concepts in previous studies.

\begin{flushleft}

\end{flushleft}

\begin{flushleft}
\end{flushleft}

\end{abstract}


\section{Introduction}

A \textit{delay differential equation (DDE)} is a differential equation where the time derivative of the unknown function $x$ depends on the past information of $x$.
When such a past dependence in a DDE is expressed as the dependence of $\Dot{x}(t)$ on $x|_{\sbra*{t - r, t}}$ with a given constant $r > 0$, the dynamics concept of the DDE can be understood as the time evolution of $x_t$.
Here $x_t \colon \sbra*{-r, 0} \to \mathbb{K}^n$ is a function defined by
\begin{equation*}
	x_t(\theta) \coloneqq x\rbra*{t + \theta}
	\mspace{25mu}
	(\theta \in \sbra*{-r, 0}),
\end{equation*}
which is called the \textit{history segment} of $x$ at $t$.
Throughout this paper, let $n \ge 1$ be an integer, $\mathbb{K} = \mathbb{R}$ or $\mathbb{C}$, and $r > 0$ be a constant.
Since $x_t$ is a function on the interval $\sbra*{-r, 0}$, we need to choose a function space in which the history segment $x_t$ lives in order to consider the time evolution of $x_t$.
We call such a function space a \textit{history function space} of the DDE.

As a history function space, it is usual to choose the Banach space $C \coloneqq C\rbra*{\sbra*{-r, 0}, \mathbb{K}^n}$ of continuous functions from $\sbra*{-r, 0}$ to $\mathbb{K}^n$ with the supremum norm given by
\begin{equation*}
	\norm*{\phi}
	\coloneqq \sup_{\theta \in \sbra*{-r, 0}} \abs*{\phi(\theta)}
	\mspace{25mu}
	(\text{$\abs*{\argdot}$ is a fixed norm on $\mathbb{K}^n$})
\end{equation*}
for $\phi \in C$.
Viewing the dynamics of DDEs as the time evolution of the history segment in the Banach space $C$ goes back to Krasovskii~\cite{Krasovskii_1963}.
See also \cite{Hale_1963_stability}, \cite{Hale_2006_history} as reviews of this matter.

With this point of view, DDEs are formulated as \textit{retarded functional differential equations (RFDEs)}, where a linear RFDE has the form
\begin{equation}\label{eq: linear RFDE}
	\Dot{x}(t) = Lx_t
	\mspace{25mu}
	(t \ge 0)
\end{equation}
for a given continuous linear map $L \colon C \to \mathbb{K}^n$.
The usual notion of solutions to the linear RFDE~\eqref{eq: linear RFDE} is defined under an initial condition $x_0 = \phi \in C$, where the function $\phi$ is called the \textit{initial history function}.
Then a continuous function $x \colon [-r, \infty) \to \mathbb{K}^n$ is a \textit{solution} of \eqref{eq: linear RFDE} under the initial condition $x_0 = \phi \in C$ if and only if it satisfies both the initial condition and an integral equation
\begin{equation}\label{eq: integral eq for linear RFDE}
	x(t) = \phi(0) + \int_0^t Lx_s \mspace{2mu} \dd s
	\mspace{25mu}
	(t \ge 0).
\end{equation}
We refer the reader to \cite{Hale_1977}, \cite{Hale--VerduynLunel_1993} as general references of the theory of RFDEs.

The above solution concept does not allow us to choose an initial history function from discontinuous functions.
Such a need can be seen in various perspectives, e.g., in the problem of the choice of a history function space as a Hilbert space, and in the problem of  a definition of the fundamental matrix solution of linear RFDEs.
For the former problem, there have been extensive researches from the mid 1970s, represented by e.g., Delfour and Mitter~\cite{Delfour--Mitter_1972_hereditary}, Bernier and Manitius~\cite{Bernier--Manitius_1978}, Delfour~\cite{Delfour_1980}, Delfour and Manitius~\cite{Delfour--Manitius_1980}.
See also the references therein.
However, the understanding of the solution concepts in the above references are not straightforward.

Compared with these, there is a straightforward way to accomplish a solution concept under discontinuous history functions.
This is the notion of \textit{mild solutions} introduced in \cite{Nishiguchi_2023}, which is obtained by replacing the integral $\int_0^t Lx_s \mspace{2mu} \dd s$ in \eqref{eq: integral eq for linear RFDE} by $L\int_0^t x_s \mspace{2mu} \dd s$.
Here for each $x \in \mathcal{L}^1_\mathrm{loc}\rbra*{[-r, \infty), \mathbb{K}^n}$, $\int_0^t x_s \mspace{2mu} \dd s \in C$ is defined by
\begin{equation*}
	\rbra*{ \int_0^t x_s \mspace{2mu} \dd s }(\theta)
	\coloneqq \int_0^t x\rbra*{s + \theta} \mspace{2mu} \dd s
	\mspace{25mu}
	(\theta \in \sbra*{-r, 0}).
\end{equation*}
In this paper, for each interval $J \subset \mathbb{R}$ and each $p \in [1, \infty)$, let $\mathcal{L}^p_\mathrm{loc}\rbra*{J, \mathbb{K}^n}$ denote the set of all locally Lebesgue $p$-integrable functions from $J$ to $\mathbb{K}^n$ defined almost everywhere.
Furthermore, let $\mathcal{L}^\infty_\mathrm{loc}\rbra*{J, \mathbb{K}^n}$ denote the set of all locally essentially bounded functions from $J$ to $\mathbb{K}^n$ defined almost everywhere.

The purpose of this paper is to reveal a connection between the solution concepts in \cite{Delfour--Mitter_1972_hereditary}, \cite{Bernier--Manitius_1978}, \cite{Delfour_1980}, \cite{Delfour--Manitius_1980} and the notion of mild solutions in \cite{Nishiguchi_2023} by showing a regularity property of mild solutions to the linear RFDE~\eqref{eq: linear RFDE}.
We now state the precise definition of mild solutions given in \cite[Definition~2.5]{Nishiguchi_2023}.
We will use the following notation.

\begin{notation}[cf.\ \cite{Delfour--Mitter_1972_hereditary}]
For each $p \in \sbra*{1, \infty}$, let
\begin{equation*}
	\mathcal{M}^p\rbra*{\sbra*{-r, 0}, \mathbb{K}^n}
	\coloneqq \Set*{\phi \in \mathcal{L}^p\rbra*{\sbra*{-r, 0}, \mathbb{K}^n}}{\text{$\phi$ is defined at $0$}}.
\end{equation*}
$\mathcal{M}^p\rbra*{\sbra*{-r, 0}, \mathbb{K}^n}$ will be abbreviated as $\mathcal{M}^p$.
\end{notation}

\begin{definition}[\cite{Nishiguchi_2023}]\label{dfn: mild sol}
Let $\phi \in \mathcal{M}^1$ be given.
We call a function $x \in \mathcal{L}^1_\mathrm{loc}\rbra*{[-r, \infty), \mathbb{K}^n}$ a \textit{mild solution} of the linear RFDE~\eqref{eq: linear RFDE} under the initial condition $x_0 = \phi$ if (i) $x$ is defined on $[0, \infty)$, and (ii) a system of equations
\begin{equation}\label{eq: mild sol}
	\left\{
	\begin{alignedat}{2}
		x(t) &= \phi(0) + L\int_0^t x_s \mspace{2mu} \dd s
		&\mspace{25mu}&
		(t \ge 0), \\
		x(t) &= \phi(t)
		& &
		(\text{a.e.\ $t \in \sbra*{-r, 0}$})
	\end{alignedat}
	\right.
\end{equation}
is satisfied.
\end{definition}

We note that the property that $\phi$ is defined at $0$ is necessary to give the initial value $\phi(0)$ in Definition~\ref{dfn: mild sol}.
Based on this definition, one can show that a mild solution of \eqref{eq: linear RFDE} under an initial condition $x_0 = \phi \in \mathcal{M}^1$ coincides with the usual solution if $\phi \in C$.
We also have the existence and uniqueness of a mild solution of \eqref{eq: linear RFDE} under each initial condition $x_0 = \phi \in \mathcal{M}^1$, where the unique mild solution is denoted by
\begin{equation*}
	x\rbra*{\argdot; \phi} \colon \dom(\phi) \cup [0, \infty) \to \mathbb{K}^n.
\end{equation*}
See \cite[Subsections~2.1, 2.2, and 2.3]{Nishiguchi_2023} for the details.
It holds the mild solution $x\rbra*{\argdot; \phi}$ of \eqref{eq: linear RFDE} under $x_0 = \phi \in \mathcal{M}^1$ is continuous on $[0, \infty)$ because the function
\begin{equation}\label{eq: t mapsto int_0^t x_s ds in C}
	[0, \infty) \ni t \mapsto \int_0^t x_s \mspace{2mu} \dd s \in C
\end{equation}
is continuous for any $x \in \mathcal{L}^1_\mathrm{loc}\rbra*{[-r, \infty), \mathbb{K}^n}$ (see \cite[Lemma~2.8]{Nishiguchi_2023}), which implies that the function
\begin{equation}\label{eq: t mapsto L int_0^t x_s ds}
	[0, \infty) \ni t \mapsto L\int_0^t x_s \mspace{2mu} \dd s \in \mathbb{K}^n
\end{equation}
is also continuous because of the continuity of $L$.

For a special class of initial history functions, we have a more nicer regularity property of mild solutions.
This class is given by the set of instantaneous inputs of vectors in $\mathbb{K}^n$.
Here an \textit{instantaneous input} $\Hat{\xi} \colon \sbra*{-r, 0} \to \mathbb{K}^n$ of a vector $\xi \in \mathbb{K}^n$ is defined by
\begin{equation*}
	\Hat{\xi}(\theta) \coloneqq
	\begin{cases}
		0 & (\theta \in [-r, 0)) \\
		\xi & (\theta = 0).
	\end{cases}
\end{equation*}
For any $\xi \in \mathbb{K}^n$, the instantaneous input $\Hat{\xi}$ belongs to $\mathcal{M}^1$.
Therefore, one can consider the mild solution $x\rbra[\big]{\argdot; \Hat{\xi}} \colon [-r, \infty) \to \mathbb{K}^n$ of \eqref{eq: linear RFDE} under $x_0 = \Hat{\xi}$.

For this class of mild solutions, the following result holds (cf.\ \cite[Theorems~3.10 and 3.5]{Nishiguchi_2023}).

\begin{theorem}[cf.\ \cite{Nishiguchi_2023}]\label{thm: regularity and DE of x(argdot; Hat{xi})}
For each $\xi \in \mathbb{K}^n$, the mild solution $x\rbra[\big]{\argdot; \Hat{\xi}}$ of \eqref{eq: linear RFDE} under $x_0 = \Hat{\xi}$ is locally Lipschitz continuous and differentiable almost everywhere on $[0, \infty)$.
Furthermore, $x \coloneqq x\rbra[\big]{\argdot; \Hat{\xi}}$ satisfies a differential equation
\begin{equation}\label{eq: DE of x(argdot; Hat{xi})}
	\Dot{x}(t) = \int_{-t}^0 \dd \eta(\theta) \mspace{2mu} x\rbra*{t + \theta}
\end{equation}
for almost all $t \ge 0$.
\end{theorem}

For the statement, $\eta \colon (-\infty, 0] \to M_n\rbra*{\mathbb{K}}$ is a matrix-valued function with the properties that $\eta$ is of bounded variation on $\sbra*{-r, 0}$, $\eta$ is constant on $(-\infty, -r]$, and
\begin{equation}\label{eq: expression of L}
	L\phi = \int_{-r}^0 \dd \eta(\theta) \mspace{2mu} \phi(\theta)
\end{equation}
holds for all $\phi \in C$.
Here the right-hand side is a Riemann--Stieltjes integral of a vector-valued function $\phi$ with respect to the matrix-valued function $\eta$.
The existence of such an $\eta$ is ensured by a corollary of the Riesz representation theorem.
See \cite[Appendix~A]{Nishiguchi_2023} for Riemann--Stieltjes integrals for matrix-valued functions.
We note that \eqref{eq: DE of x(argdot; Hat{xi})} is not a differential equation of \textit{infinite retardation} because the right-hand side of \eqref{eq: DE of x(argdot; Hat{xi})} becomes
\begin{equation*}
	\int_{-r}^0 \dd \eta(\theta) \mspace{2mu} x\rbra*{t + \theta}
	= Lx_t
\end{equation*}
for all $t \in [r, \infty)$.

For the proof of Theorem~\ref{thm: regularity and DE of x(argdot; Hat{xi})}, the Riemann--Stieltjes convolution plays a fundamental role.
Here for a continuous function $f \colon [0, \infty) \to M_n\rbra*{\mathbb{K}}$ and a function $\alpha \colon [0, \infty) \to M_n\rbra*{\mathbb{K}}$ of locally bounded variation (i.e., of bounded variation on any closed and bounded interval of $[0, \infty)$), the function $\dd \alpha * f \colon [0, \infty) \to M_n\rbra*{\mathbb{K}}$ is defined by
\begin{equation*}
	\rbra*{\dd \alpha * f}(t) \coloneqq \int_0^t \dd \alpha(u) \mspace{2mu} f\rbra*{t - u}
	\mspace{25mu}
	(t \ge 0).
\end{equation*}
Since the right-hand side is a Riemann--Stieltjes integral, the function $\dd \alpha * f$ is called a \textit{Riemann--Stieltjes convolution}.
It holds that the function $\dd \alpha * f$ is a sum of a continuous function and a function of locally bounded variation (e.g., see \cite[Theorem~3.5]{Nishiguchi_2023}).
Therefore, $\dd \alpha * f$ is Riemann integrable on any closed and bounded interval of $[0, \infty)$.
By defining a function $\Check{\eta} \colon [0, \infty) \to M_n\rbra*{\mathbb{K}}$ by
\begin{equation}\label{eq: Check{eta}}
	\Check{\eta}(t) \coloneqq -\eta\rbra*{-t}
	\mspace{25mu}
	(t \ge 0),
\end{equation}
the differential equation~\eqref{eq: DE of x(argdot; Hat{xi})} can be expressed as $\Dot{x}(t) = \rbra*{ \dd \Check{\eta} * x|_{[0, \infty)} }(t)$.

In this paper, we give an extension of Theorem~\ref{thm: regularity and DE of x(argdot; Hat{xi})}.
The following is the main result of this paper.

\begin{theorem}\label{thm: main result}
For every $\phi \in \mathcal{M}^p$, where $p \in \sbra*{1, \infty}$, the mild solution $x\rbra*{\argdot; \phi}$ of the linear RFDE~\eqref{eq: linear RFDE} under $x_0 = \phi$ is locally absolutely continuous on $[0, \infty)$.
Furthermore, the following statements hold:
\begin{enumerate}
\item $x\rbra*{\argdot; \phi}|_{[0, \infty)} \colon [0, \infty) \to \mathbb{K}^n$ is differentiable almost everywhere and its derivative belongs to $\mathcal{L}^p_\mathrm{loc}\rbra*{[0, \infty), \mathbb{K}^n}$.
\item There exists a function $f\rbra*{\argdot; \phi} \in \mathcal{L}^p_\mathrm{loc}\rbra*{[0, \infty), \mathbb{K}^n}$ vanishing on $[r, \infty)$ such that
\begin{equation}\label{eq: DE of mild sol}
	\Dot{x}\rbra*{t; \phi}
	= \int_{-t}^0 \dd \eta(\theta) \mspace{2mu} x\rbra*{t + \theta; \phi} + f\rbra*{t; \phi}
	\mspace{25mu}
	(\text{a.e.\ $t \ge 0$})
\end{equation}
holds.
\end{enumerate}
\end{theorem}

For an interval $J \subset \mathbb{R}$, a function $f \colon J \to \mathbb{K}^n$ is said to be \textit{locally absolutely continuous} if $f|_K \colon K \to \mathbb{K}^n$ is absolutely continuousfor any closed and bounded interval $K$ contained in $J$.
Let $\AC_\mathrm{loc}\rbra*{J, \mathbb{K}^n}$ denote the set of all locally absolutely continuous functions from $J$ to $E$.

We give a comment on an extension of Theorem~\ref{thm: main result} to the case that $\mathbb{K}^n$ is replaced with an infinite-dimensional Banach space $E$.
Such an extension is of course natural, however, some additional assumptions should be imposed.
On a condition on $L$, we need to assume the existence of an operator-valued function $\eta \colon (-\infty, 0] \to \mathcal{B}(E)$ such that $\eta|_{\sbra*{-r, 0}} \colon \sbra*{-r, 0} \to \mathcal{B}(E)$ is of strong bounded variation, $\eta$ is constant on $(-\infty, -r]$, and \eqref{eq: expression of L} holds for all $\psi \in C$.
The above integral is a Riemann--Stieltjes integral of a vector-valued function with respect to an operator-valued function.
See \cite[Section~2]{Diekmann--Gyllenberg--Thieme_1993} for details of this type of integrals.
For the above mentioned extension, the extent to which there is a difference between the boundedness of $L$ and the above assumption on $L$ should be discussed.

\vspace{0.5\baselineskip}

This paper is organized as follows.
The situation of the statement and its proof of Theorem~\ref{thm: main result} are different depending on the cases $p = 1$, $p = \infty$, and $1 < p < \infty$.
In Section~\ref{sec: p = 1 or infty}, we give proof of Theorem~\ref{thm: main result} for $p = 1$ or $p = \infty$.
For the proof for $p = 1$, a difficulty is that the function~\eqref{eq: t mapsto int_0^t x_s ds in C} is not necessarily locally absolutely continuous for $x \in \mathcal{L}^1_\mathrm{loc}\rbra*{[-r, \infty), \mathbb{K}^n}$.
In Subsection~\ref{subsec: local absolute continuity}, we resolve this difficulty by directly showing the local absolute continuity of the function~\eqref{eq: t mapsto L int_0^t x_s ds}.
In Subsection~\ref{subsec: differential equations}, we obtain the differential equation~\eqref{eq: DE of mild sol} by adopting methods used in \cite{Nishiguchi_2023}.
Contrary to the case $p = 1$, the case $p = \infty$ is easier to handle because one can show that the function~\eqref{eq: t mapsto int_0^t x_s ds in C} is locally Lipschitz continuous for $x \in \mathcal{L}^\infty_\mathrm{loc}\rbra*{[-r, \infty), \mathbb{K}^n}$.
Subsection~\ref{subsec: proof for p = 1 or infty} is on the proof of Theorem~\ref{thm: main result} for the cases $p = 1$ or $p = \infty$.
In Section~\ref{sec: 1 < p < infty}, we give proof of Theorem~\ref{thm: main result} for $1 < p < \infty$ by adopting a density argument used by Delfour and Manitius~\cite{Delfour--Manitius_1980}.
Finally, in Section~\ref{sec: discussion}, we see a connection between the notion of mild solutions and the solution concept used by Delfour and Manitius~\cite{Delfour--Manitius_1980}.
In Appendix~\ref{sec: RS integral}, we collect results on Riemann--Stieltjes integrals.

\section{Proof of the main result for \texorpdfstring{$p = 1$}{p = 1} or \texorpdfstring{$p = \infty$}{p = infty}}\label{sec: p = 1 or infty}

In this section, we show Theorem~\ref{thm: main result} for $p = 1$ or $p = \infty$.
The proof is divided into the following two parts for each $\phi \in \mathcal{M}^1$:
\begin{itemize}
\item Proof of the local absolute continuity of $x\rbra*{\argdot; \phi}|_{[0, \infty)}$.
\item To obtain the differential equation~\eqref{eq: DE of mild sol} which $x\rbra*{\argdot; \phi}|_{[0, \infty)}$ satisfies.
\end{itemize}

\subsection{Local absolute continuity}\label{subsec: local absolute continuity}

We use the following notation.

\begin{notation}
Let $\sbra*{a, b}$ be a closed and bounded interval of $\mathbb{R}$.
For each matrix-valued function $\alpha \colon \sbra*{a, b} \to M_n\rbra*{\mathbb{K}}$ of bounded variation with respect to the operator norm $\abs*{\argdot}$ on $M_n\rbra*{\mathbb{K}}$, let $V_\alpha \colon \sbra*{a, b} \to \mathbb{R}$ denote its total variation function.
Namely,
\begin{equation*}
	V_\alpha(t) \coloneqq \Var\rbra*{\alpha|_{\sbra*{a, t}}}
\end{equation*}
holds for all $t \in \sbra*{a, b}$.
Here $\Var(\alpha)$ denotes the total variation of $\alpha$, and we interpret $V_\alpha(a) = 0$.
\end{notation}

We refer the reader to \cite[Chapter~9]{Shapiro_2018} for a reference of scalar-valued functions of bounded variation.
The following is a key result.

\begin{theorem}\label{thm: loc AC of t mapsto L int_0^t x_s ds}
For any $x \in \mathcal{L}^1_\mathrm{loc}\rbra*{[-r, \infty), \mathbb{K}^n}$, the function~\eqref{eq: t mapsto L int_0^t x_s ds}
\begin{equation*}
	[0, \infty) \ni t \mapsto L\int_0^t x_s \mspace{2mu} \dd s \in \mathbb{K}^n
\end{equation*}
is locally absolutely continuous.
\end{theorem}

The following lemma is necessary for the proof of Theorem~\ref{thm: loc AC of t mapsto L int_0^t x_s ds}.
The corresponding statement for scalar-valued case is given in \cite[Theorem~5b of Chapter I]{Widder_1941} without proof .

\begin{lemma}\label{lem: sharp estimate of RS integral}
Let $\sbra*{a, b}$ be a closed and bounded interval of $\mathbb{R}$.
Let $f \colon \sbra*{a, b} \to M_n\rbra*{\mathbb{K}}$ be continuous and $\alpha \colon \sbra*{a, b} \to M_n\rbra*{\mathbb{K}}$ be of bounded variation.
Then
\begin{equation}\label{eq: sharp estimate of RS integral}
	\abs*{\int_a^b \dd \alpha(t) \mspace{2mu} f(t)}
	\le \int_a^b \abs*{f(t)} \mspace{2mu} \dd V_\alpha(t)
\end{equation}
holds.
Here the right-hand side is the Riemann--Stieltjes integral of the real-valued continuous function $\sbra*{a, b} \ni t \mapsto \abs*{f(t)} \in \mathbb{R}$ with respect to the monotonically increasing function $V_\alpha$.
\end{lemma}

\begin{proof}
As in scalar-valued case,
\begin{equation*}
	\Var\rbra*{\alpha|_{\sbra*{a, c}}} + \Var\rbra*{\alpha|_{\sbra*{c, b}}}
	= \Var(\alpha)
\end{equation*}
holds for any $c \in \rbra*{a, b}$.
Therefore, for any subinterval $\sbra*{c, d} \subset \sbra*{a, b}$,
\begin{equation}\label{eq: estimate for total variation function}
	\abs*{\alpha(d) - \alpha(c)}
	\le \Var\rbra*{\alpha|_{\sbra*{c, d}}}
	= V_\alpha(d) - V_\alpha(c)
\end{equation}
holds.

Let $\rbra*{t_k}_{k = 0}^m$ give a partition of $\sbra*{a, b}$ and $\rbra*{\tau_k}_{k = 1}^m$ be given so that $t_{k - 1} \le \tau_k \le t_k$ holds for each $1 \le k \le m$.
Then we have
\begin{align*}
	\abs*{ \sum_{k = 1}^m \sbra*{ \alpha\rbra*{t_k} - \alpha\rbra*{t_{k - 1}} }f\rbra*{\tau_k} }
	&\le \sum_{k = 1}^m \abs*{\alpha\rbra*{t_k} - \alpha\rbra*{t_{k - 1}}}\abs*{f\rbra*{\tau_k}} \\
	&\le \sum_{k = 1}^m \mspace{2mu} \abs*{f\rbra*{\tau_k}} \sbra*{ V_\alpha\rbra*{t_k} - V_\alpha\rbra*{t_{k - 1}} },
\end{align*}
where \eqref{eq: estimate for total variation function} is used.
By taking the limit as $\max_{1 \le k \le m} \rbra*{t_k - t_{k - 1}} \to 0$, we obtain \eqref{eq: sharp estimate of RS integral}.
\end{proof}

\begin{proof}[Proof of Theorem~\ref{thm: loc AC of t mapsto L int_0^t x_s ds}]
For a given $x \in \mathcal{L}^1_\mathrm{loc}\rbra*{[-r, \infty), \mathbb{K}^n}$, we define $y \in \AC_\mathrm{loc}\rbra*{[-r, \infty), \mathbb{K}^n}$ by
\begin{equation}\label{eq: indefinite integral of x}
	y(t) \coloneqq \int_0^t x(s) \mspace{2mu} \dd s
	\mspace{25mu}
	(t \in [-r, \infty)).
\end{equation}
Then we have
\begin{equation*}
	\rbra*{\int_0^t x_s \mspace{2mu} \dd s}(\theta)
	= \int_\theta^{t + \theta} x(s) \mspace{2mu} \dd s
	= y\rbra*{t + \theta} + \int_\theta^0 x(s) \mspace{2mu} \dd s
\end{equation*}
for all $t \ge 0$ and $\theta \in \sbra*{-r, 0}$, which yields
\begin{equation*}
	L\int_0^t x_s \mspace{2mu} \dd s
	= Ly_t + \int_{-r}^0 \dd \eta(\theta) \mspace{2mu} \rbra*{\int_\theta^0 x(s) \mspace{2mu} \dd s}
\end{equation*}
for all $t \ge 0$.
Therefore, the local absolute continuity of \eqref{eq: t mapsto L int_0^t x_s ds} is reduced to that of $[0, \infty) \ni t \mapsto Ly_t \in \mathbb{K}^n$.

We fix $T > 0$ and show that $\sbra*{0, T} \ni t \mapsto Ly_t \in \mathbb{K}^n$ is absolutely continuous.
Let $\ep > 0$ be given.
By the absolute continuity of $y|_{\sbra*{-r, T}}$, one can choose a $\delta > 0$ with the following property: for any pairwise disjoint finite open intervals $\rbra*{s_1, t_1}, \dots, \rbra*{s_m, t_m}$ in $\sbra*{-r, T}$, $\sum_{j = 1}^m \rbra*{t_j - s_j} < \delta$ implies
\begin{equation*}
	\sum_{j = 1}^m \abs*{y(t_j) - y(s_j)} < \ep.
\end{equation*}
Let $\rbra*{s_1, t_1}, \dots, \rbra*{s_m, t_m}$ be pairwise disjoint finite open intervals in the interval $\sbra*{0, T}$ with $\sum_{j = 1}^m \rbra*{t_j - s_j} < \delta$.
From Lemma~\ref{lem: sharp estimate of RS integral}, we have
\begin{align*}
	\sum_{j = 1}^m \abs*{Ly_{t_j} - Ly_{s_j}}
	&= \sum_{j = 1}^m \abs*{ \int_{-r}^0 \dd \eta(\theta) \mspace{2mu} [y\rbra*{t_j + \theta} - y\rbra*{s_j + \theta}] } \\
	&\le \sum_{j = 1}^m \int_{-r}^0 \abs*{y\rbra*{t_j + \theta} - y\rbra*{s_j + \theta}} \mspace{2mu} \dd V_\eta(\theta),
\end{align*}
where the last term is equal to $\int_{-r}^0 \sum_{j = 1}^m \abs*{y\rbra*{t_j + \theta} - y\rbra*{s_j + \theta}} \mspace{2mu} \dd V_\eta(\theta)$.
Since for each $\theta \in \sbra*{-r, 0}$,
\begin{equation*}
	\rbra*{s_1 + \theta, t_1 + \theta}, \dots, \rbra*{s_m + \theta, t_m + \theta}
\end{equation*}
are pairwise disjoint finite open intervals in $\sbra*{-r, T}$ with
\begin{equation*}
	\sum_{j = 1}^m \sbra*{\rbra*{t_j + \theta} - \rbra*{s_j + \theta}}
	= \sum_{j = 1}^m \rbra*{t_j - s_j}
	< \delta,
\end{equation*}
an estimate
\begin{equation*}
	\sum_{j = 1}^m \abs*{Ly_{t_j} - Ly_{s_j}}
	\le \Var(\eta) \cdot \ep
\end{equation*}
is obtained.
This shows the absolute continuity of $\sbra*{0, T} \ni t \mapsto Ly_t \in \mathbb{K}^n$.
\end{proof}

\begin{remark}
For a given $y \in \AC_\mathrm{loc}\rbra*{[-r, \infty), \mathbb{K}^n}$, the function
\begin{equation}
	[0, \infty) \ni t \mapsto y_t \in C
\end{equation}
is not necessarily locally absolutely continuous.
The reason is that one cannot change the order of summation and supremum in
\begin{equation*}
	\sum_{j = 1}^m \sup_{\theta \in \sbra*{-r, 0}} \abs*{y\rbra*{t_j + \theta} - y\rbra*{s_j + \theta}}.
\end{equation*}
Here $\rbra*{s_1, t_1}, \dots, \rbra*{s_m, t_m}$ are pairwise disjoint finite open intervals.
\end{remark}

\begin{theorem}\label{thm: loc AC of mild sol}
For any $\phi \in \mathcal{M}^1$, the mild solution $x(\argdot; \phi) \colon \dom(\phi) \cup [0, \infty) \to \mathbb{K}^n$ of \eqref{eq: linear RFDE} under $x_0 = \phi$ is locally absolutely continuous on $[0, \infty)$.
\end{theorem}

\begin{proof}
By definition, $x \coloneqq x(\argdot; \phi)$ satisfies
\begin{equation*}
	x(t) = \phi(0) + L\int_0^t x_s \mspace{2mu} \dd s
\end{equation*}
for all $t \ge 0$.
Since the function $x$ belongs to $\mathcal{L}^1_\mathrm{loc}\rbra*{[-r, \infty), \mathbb{K}^n}$, it holds that the right-hand side of the above equation is locally absolute continuous with respect to $t \in [0, \infty)$ from Theorem~\ref{thm: loc AC of t mapsto L int_0^t x_s ds}.
\end{proof}

\subsection{Differential equations for mild solutions}\label{subsec: differential equations}

We use the following notations.

\begin{notation}
For each $\phi \in \mathcal{M}^1$, we define a function $\Bar{\phi} \colon \dom(\phi) \cup [0, \infty) \to \mathbb{K}^n$ by
\begin{equation*}
	\Bar{\phi}(t) \coloneqq
	\begin{cases}
		\phi(t) & (t \in \dom(\phi)) \\
		\phi(0) & (t \in [0, \infty)),
	\end{cases}
\end{equation*}
which is called the \textit{static prolongation} of $\phi$.
\end{notation}

\begin{notation}[\cite{Nishiguchi_2023}]
For each $\phi \in \mathcal{M}^1$, we define a function $G\rbra*{\argdot; \phi} \colon [0, \infty) \to \mathbb{K}^n$ by
\begin{equation*}
	G\rbra*{t; \phi}
	\coloneqq \int_{-r}^0 \dd \eta(\theta) \mspace{2mu} \rbra*{\int_\theta^0 \phi(s) \mspace{2mu} \dd s} + \int_{-r}^{-t} \dd \eta(\theta) \mspace{2mu} \rbra*{\int_0^{t + \theta} \Bar{\phi}(s) \mspace{2mu} \dd s}
	\mspace{25mu}
	(t \ge 0).
\end{equation*}
Since $\eta$ is constant on $(-\infty, -r]$, the function $G\rbra*{\argdot; \phi}$ is constant on $[r, \infty)$.
\end{notation}

\begin{notation}
For each locally Riemann integrable function $h \colon [0, \infty) \to M_n\rbra*{\mathbb{K}}$, we define a function $\mathcal{V}h \colon [0, \infty) \to M_n\rbra*{\mathbb{K}}$ by
\begin{equation*}
	\rbra*{\mathcal{V}h}(t)
	\coloneqq \int_0^t h(s) \mspace{2mu} \dd s
	\mspace{25mu}
	(t \in [0, \infty)),
\end{equation*}
where the integral is a Riemann integral.
We call $\mathcal{V}$ the \textit{Volterra operator}.
\end{notation}

In this subsection, we show the following theorem.

\begin{theorem}\label{thm: DE satisfied by mild sol}
Let $\phi \in \mathcal{M}^1$ be given.
Then $x\rbra*{\argdot; \phi}|_{[0, \infty)}$ is differentiable almost everywhere, and there exists a function $f\rbra*{\argdot; \phi} \in \mathcal{L}^1_\mathrm{loc}\rbra*{[0, \infty), \mathbb{K}^n}$ vanishing at $[r, \infty)$ such that \eqref{eq: DE of mild sol}
\begin{equation*}
	\Dot{x}\rbra*{t; \phi} = \int_{-t}^0 \dd \eta(\theta) \mspace{2mu} x\rbra*{t + \theta; \phi} + f\rbra*{t; \phi}
	\mspace{25mu}
	(\text{a.e.\ $t \ge 0$})
\end{equation*}
holds.
\end{theorem}

\begin{proof}
Let $x \coloneqq x\rbra*{\argdot; \phi}|_{[0, \infty)}$.
As in \cite[Section~6.2]{Nishiguchi_2023}, we have
\begin{equation}\label{eq: Eq of mild sol}
	x(t)
	= \phi(0) + \rbra*{\dd \Check{\eta} * \mathcal{V}x}(t) + G\rbra*{t; \phi}
\end{equation}
for all $t \ge 0$.
Here $x \in \AC_\mathrm{loc}\rbra*{[0, \infty), \mathbb{K}^n}$ holds from Theorem~\ref{thm: loc AC of mild sol}.
Furthermore,
\begin{equation*}
	\dd \Check{\eta} * \mathcal{V}x = \mathcal{V}\rbra*{\dd \Check{\eta} * x}
\end{equation*}
holds from \cite[Theorem~3.7]{Nishiguchi_2023}, where $\dd \Check{\eta} * x$ is a locally Riemann integrable function.
Therefore, \eqref{eq: Eq of mild sol} yields that $G\rbra*{\argdot; \phi} \in \AC_\mathrm{loc}([0, \infty), \mathbb{K}^n)$ holds.
It also implies that $G\rbra*{\argdot; \phi}$ is differentiable almost everywhere, whose derivative belongs to $\mathcal{L}^1_\mathrm{loc}\rbra*{[0, \infty), \mathbb{K}^n}$.
Thus, differentiating both sides of \eqref{eq: Eq of mild sol}, we have
\begin{equation*}
	\Dot{x}(t)
	= \rbra*{\dd \Check{\eta} * x}(t) + \Dot{G}\rbra*{t; \phi}
	\mspace{25mu}
	(\text{a.e.\ $t \ge 0$}).
\end{equation*}
Since $G\rbra*{\argdot; \phi}$ is constant on $[r, \infty)$, \eqref{eq: DE of mild sol} is obtained with $f\rbra*{\argdot; \phi} = \Dot{G}\rbra*{\argdot; \phi}$.
\end{proof}

\subsection{Proof of Theorem~\ref{thm: main result} for \texorpdfstring{$p = 1$}{p = 1} or \texorpdfstring{$p = \infty$}{p = infty}}\label{subsec: proof for p = 1 or infty}

The proof of Theorem~\ref{thm: main result} for $p = 1$ is obtained by combining the proofs of Theorems~\ref{thm: loc AC of mild sol} and \ref{thm: DE satisfied by mild sol}.
We now give the proof of Theorem~\ref{thm: main result} for $p = \infty$.
For the proof, the following proposition is used.

\begin{proposition}\label{prop: loc Lip of t mapsto int_0^t x_s ds in C}
For any $x \in \mathcal{L}^\infty_\mathrm{loc}\rbra*{[-r, \infty), \mathbb{K}^n}$, the function~\eqref{eq: t mapsto int_0^t x_s ds in C}
\begin{equation*}
	[0, \infty) \ni t \mapsto \int_0^t x_s \mspace{2mu} \dd s \in C
\end{equation*}
is locally Lipschitz continuous.
Consequently, the function~\eqref{eq: t mapsto L int_0^t x_s ds}
\begin{equation*}
	[0, \infty) \ni t \mapsto L\int_0^t x_s \mspace{2mu} \dd s \in \mathbb{K}^n
\end{equation*}
is also locally Lipschitz continuous.
\end{proposition}

\begin{proof}
By defining a locally Lipschitz continuous function $y \colon [-r, \infty) \to \mathbb{K}^n$ by \eqref{eq: indefinite integral of x}, it is sufficient to show the local Lipschitz continuity of $[0, \infty) \ni t \mapsto y_t \in C$.

Let $T > 0$ be fixed.
The Lipschitz continuity of $y|_{\sbra*{-r, T}}$ yields that
\begin{equation*}
	\norm*{y_{t_1} - y_{t_2}}
	= \sup_{\theta \in \sbra*{-r, 0}} \abs*{y(t_1 + \theta) - y(t_2 + \theta)} \\
	\le \lip\rbra*{ y|_{\sbra*{-r, T}} } \cdot \abs*{t_1 - t_2}
\end{equation*}
holds for all $t_1, t_2 \in \sbra*{0, T}$.
Here $\lip\rbra*{ y|_{\sbra*{-r, T}} }$ denotes the (best) Lipschitz constant of the function $y|_{\sbra*{-r, T}}$.
The above argument shows the local Lipschitz continuity.
\end{proof}

\begin{proof}[Proof of Theorem~\ref{thm: main result} for $p = \infty$]
As in the proof of Theorem~\ref{thm: loc AC of mild sol}, $x\rbra*{\argdot; \phi}|_{[0, \infty)}$ is locally Lipschitz continuous from Proposition~\ref{prop: loc Lip of t mapsto int_0^t x_s ds in C}.
By combining with this property, \eqref{eq: Eq of mild sol} yields that $G\rbra*{\argdot; \phi}$ is also locally Lipschitz continuous.
This implies that $G\rbra*{\argdot; \phi}$ is differentiable almost everywhere, whose derivative belongs to $\mathcal{L}^\infty_\mathrm{loc}\rbra*{[0, \infty), \mathbb{K}^n}$.
Therefore, the statements of Theorem~\ref{thm: main result} for $p = \infty$ are obtained with $f\rbra*{\argdot; \phi} = \Dot{G}\rbra*{\argdot; \phi}$ as the proof of Theorem~\ref{thm: DE satisfied by mild sol}.
\end{proof}

\section{Proof of the main result for \texorpdfstring{$1 < p < \infty$}{1 < p < infty}}\label{sec: 1 < p < infty}

To prove Theorem~\ref{thm: main result} for $1 < p < \infty$, we need to show that the derivative $\Dot{G}\rbra*{\argdot; \phi}$ belongs to $\mathcal{L}^p_\mathrm{loc}\rbra*{[0, \infty), \mathbb{K}^n}$ for any $\phi \in \mathcal{M}^p$.
For this purpose, we adopt a density argument used by Delfour and Manitius~\cite{Delfour--Manitius_1980}.

\subsection{Forcing term for \texorpdfstring{$\phi \in C$}{phi in C}}

We use the following notation.

\begin{notation}[\cite{Nishiguchi_2023}]
For each $\phi \in C$, we define a function $g\rbra*{\argdot; \phi} \colon [0, \infty) \to \mathbb{K}^n$ by
\begin{align*}
	g(t; \phi)
	\coloneqq \int_{-r}^{-t} \dd \eta(\theta) \mspace{2mu} \Bar{\phi}\rbra*{t + \theta}
	= \int_t^r \dd \Check{\eta}(u) \mspace{2mu} \Bar{\phi}\rbra*{t - u}
	\mspace{25mu}
	(t \ge 0).
\end{align*}
Since $\eta$ is constant on $(-\infty, -r]$, $g\rbra*{\argdot; \phi}$ vanishes on $[r, \infty)$.
Let $g_\phi \coloneqq g\rbra*{\argdot; \phi}|_{\sbra*{0, r}}$ for each $\phi \in C$.
\end{notation}

The following result is obtained in \cite[Lemma~6.8 and Theorem~6.12]{Nishiguchi_2023}.

\begin{theorem}[\cite{Nishiguchi_2023}]\label{thm: g(argdot; phi)}
For each $\phi \in C$, the function $g\rbra*{\argdot; \phi}$ is locally Riemann integrable.
Furthermore, $G\rbra*{\argdot; \phi} = \mathcal{V}g\rbra*{\argdot; \phi}$ holds.
\end{theorem}

\begin{proof}[Another proof of Theorem~\ref{thm: g(argdot; phi)}]
Let $x \coloneqq x\rbra*{\argdot; \phi}|_{[0, \infty)}$.
Since $x\rbra*{\argdot; \phi}$ coincides with the (usual) solution of \eqref{eq: linear RFDE} under $x_0 = \phi \in C$, we have
\begin{align}
	\Dot{x}(t)
	&= \int_{-r}^0 \dd \eta(\theta) \mspace{2mu} x\rbra*{t + \theta; \phi} \notag \\
	&= \int_{-t}^0 \dd \eta(\theta) \mspace{2mu} x\rbra*{t + \theta} + \int_{-r}^{-t} \dd \eta(\theta) \mspace{2mu} \Bar{\phi}\rbra*{t + \theta} \notag \\
	&= \rbra*{ \dd \Check{\eta} * x }(t) + g\rbra*{t; \phi} \label{eq: linear RFDE with initial condition}
\end{align}
for all $t \ge 0$.
By combining the above and the continuity of $\Dot{x}|_{[0, \infty)}$, the local Riemann integrability of $g\rbra*{\argdot; \phi}$ is obtained.
By integrating the above, we have
\begin{equation*}
	x(t)
	= \phi(0) + \mathcal{V}\rbra*{ \dd \Check{\eta} * x }(t) + \rbra*{\mathcal{V}g(\argdot; \phi)}(t).
\end{equation*}
for all $t \ge 0$.
Since $\mathcal{V}\rbra*{ \dd \Check{\eta} * x } = \dd \Check{\eta} * \mathcal{V}x$ (see \cite[Theorem~3.7]{Nishiguchi_2023}), the conclusion is obtained by comparing with \eqref{eq: Eq of mild sol}.
\end{proof}

Eq.~\eqref{eq: linear RFDE with initial condition} and the definition of $g\rbra*{\argdot; \phi}$ should be compared with \cite[(2.6)]{Delfour--Manitius_1980}.

\subsection{Dependence of \texorpdfstring{$g\rbra*{\argdot; \phi}$}{g(argdot; phi)} on \texorpdfstring{$\phi$}{phi} with respect to \texorpdfstring{$L^p$}{L^p}-norm}\label{subsec: dependence of g(argdot; phi) on phi}

In this subsection, we investigate an integrability property of the function $g\rbra*{\argdot; \phi}$ for $\phi \in C$.
The following lemma will be used to obtain the integrability property.

\begin{lemma}\label{lem: f(t + theta)g(t)}
Let $f \colon \sbra*{-r, 0} \to \mathbb{R}$ be a continuous function, $\alpha \colon \sbra*{-r, 0} \to \mathbb{R}$ be a function of bounded variation, and $g \colon \sbra*{0, r} \to \mathbb{R}$ be a Riemann integrable function.
Then
\begin{equation}\label{eq: f(t + theta)g(t)}
	\int_0^r \biggl( \int_{-r}^{-t} f\rbra*{t + \theta} \mspace{2mu} \dd \alpha(\theta) \biggr) g(t) \mspace{2mu} \dd t
	= \int_{-r}^0 \biggl( \int_0^{-\theta} f\rbra*{t + \theta}g(t) \mspace{2mu} \dd t \biggr) \mspace{2mu} \dd \alpha(\theta)
\end{equation}
holds.
Here the function
\begin{equation}\label{eq: integrand in left-hand side}
	\sbra*{0, r} \ni t \mapsto \int_{-r}^{-t} f\rbra*{t + \theta} \mspace{2mu} \dd \alpha(\theta) \in \mathbb{R}
\end{equation}
is Riemann integrable, and the function
\begin{equation}\label{eq: integrand in right-hand side}
	\sbra*{-r, 0} \ni \theta \mapsto \int_0^{-\theta} f\rbra*{t + \theta}g(t) \mspace{2mu} \dd t \in \mathbb{R}
\end{equation}
is continuous.
\end{lemma}

\begin{proof}
We extend the domain of definition of $f$ to $[-r, r]$ by defining $f(t) \coloneqq f(0)$ for $t \in \sbra*{0, r}$.
Then the extended function $f \colon [-r, r] \to \mathbb{R}$ is continuous.
Let $G \colon \sbra*{0, r} \to \mathbb{R}$ be a Lipschitz continuous function defined by
\begin{equation*}
	G(t) \coloneqq \int_0^t g(s) \mspace{2mu} \dd s
	\mspace{25mu}
	(t \in \sbra*{0, r}).
\end{equation*}
Then
\begin{equation*}
	\int_0^r \rbra*{\int_{-r}^0 f\rbra*{t + \theta} \mspace{2mu} \dd \alpha(\theta)} \mspace{2mu} \dd G(t)
	= \int_{-r}^0 \rbra*{\int_0^r f\rbra*{t + \theta} \mspace{2mu} \dd G(t)} \mspace{2mu} \dd \alpha(\theta)
\end{equation*}
holds by applying Theorem~\ref{thm: Fubini thm for RS integrals of scalar-valued functions}.
By the continuity of the functions
\begin{align*}
	\sbra*{0, r} \ni t &\mapsto \int_{-r}^0 f\rbra*{t + \theta} \mspace{2mu} \dd \alpha(\theta) \in \mathbb{R}, \\
	\sbra*{0, r} \ni t &\mapsto f\rbra*{t + \theta} \in \mathbb{R},
\end{align*}
the above equality becomes (e.g., see \cite[Theorem~A.20]{Nishiguchi_2023})
\begin{equation*}
	\int_0^r \rbra*{\int_{-r}^0 f\rbra*{t + \theta} \mspace{2mu} \dd \alpha(\theta)} g(t) \mspace{2mu} \dd t
	= \int_{-r}^0 \rbra*{\int_0^r f\rbra*{t + \theta} g(t) \mspace{2mu} \dd t} \mspace{2mu} \dd \alpha(\theta).
\end{equation*}
Therefore, \eqref{eq: f(t + theta)g(t)} is obtained by showing
\begin{equation}\label{eq: f(t + theta)g(t), variant}
	\int_0^r \rbra*{\int_{-t}^0 f\rbra*{t + \theta} \mspace{2mu} \dd \alpha(\theta)} g(t) \mspace{2mu} \dd t
	= \int_{-r}^0 \rbra*{\int_{-\theta}^r f\rbra*{t + \theta} g(t) \mspace{2mu} \dd t} \mspace{2mu} \dd \alpha(\theta).
\end{equation}
We now show that the equality~\eqref{eq: f(t + theta)g(t), variant} holds.
The left-hand side and the right-hand side of \eqref{eq: f(t + theta)g(t), variant} are calculated as
\begin{equation*}
	f(0)\int_0^r \sbra*{\alpha(0) - \alpha\rbra*{-t}}g(t) \mspace{2mu} \dd t,
	\mspace{20mu}
	f(0)\int_{-r}^0 \rbra*{\int_{-\theta}^r g(s) \mspace{2mu} \dd s} \mspace{2mu} \dd \alpha(\theta),
\end{equation*}
respectively.
Here the integration by parts formula for Riemann--Stieltjes integrals yields
\begin{align*}
	\int_{-r}^0 \rbra*{\int_{-\theta}^r g(s) \mspace{2mu} \dd s} \mspace{2mu} \dd \alpha(\theta)
	&= \sbra*{ \int_{-\theta}^r g(s) \mspace{2mu} \dd s \cdot \alpha(\theta) }_{-r}^0 - \int_{-r}^0 g(-\theta) \alpha(\theta) \mspace{2mu} \dd \theta \\
	&= \int_0^r g(s)\sbra*{\alpha(0) - \alpha\rbra*{-s}} \mspace{2mu} \dd s.
\end{align*}
This shows that equality~\eqref{eq: f(t + theta)g(t), variant} holds.
The above argument also shows the Riemann integrability of \eqref{eq: integrand in left-hand side} and the continuity of \eqref{eq: integrand in right-hand side}.
\end{proof}

The following is a key lemma for the proof of Theorem~\ref{thm: main result}.

\begin{lemma}[cf.\ \cite{Delfour--Manitius_1980}]\label{lem: L^p norm of g_phi}
Let $p \in [1, \infty)$ be given.
Then for any $\phi \in C$,
\begin{equation*}
	\norm*{g_\phi}_{L^p\sbra*{0, r}}
	\le \Var(\eta)\norm*{\phi}_{L^p\sbra*{-r, 0}}
\end{equation*}
holds.
\end{lemma}

\begin{proof}
From Lemma~\ref{lem: sharp estimate of RS integral},
\begin{equation*}
	\abs*{g(t; \phi)}
	= \abs*{\int_{-r}^{-t} \dd \eta(\theta) \mspace{2mu} \phi\rbra*{t + \theta}}
	\le \int_{-r}^{-t} \abs*{\phi\rbra*{t + \theta}} \mspace{2mu} \dd V_\eta(\theta)
\end{equation*}
holds for all $t \in \sbra*{0, r}$.
Therefore, we have
\begin{equation*}
	\rbra*{ \int_0^r \abs*{g(t; \phi)}^p \mspace{2mu} \dd t }^{1/p}
	\le \sbra*{ \int_0^r \biggl( \int_{-r}^{-t} \abs*{\phi\rbra*{t + \theta}} \mspace{2mu} \dd V_\eta(\theta) \biggr)^p \mspace{2mu} \dd t }^{1/p}.
\end{equation*}
We now show that
\begin{equation}\label{eq: key inequality}
	\sbra*{ \int_0^r \biggl( \int_{-r}^{-t} \abs*{\phi\rbra*{t + \theta}} \mspace{2mu} \dd V_\eta(\theta) \biggr)^p \mspace{2mu} \dd t }^{1/p}
	\le \int_{-r}^0 \rbra*{ \int_0^{-\theta} \abs*{\phi\rbra*{t + \theta}}^p \mspace{2mu} \dd t }^{1/p} \mspace{2mu} \dd V_\eta(\theta)
\end{equation}
holds.
Let $I$ be the left-hand side of \eqref{eq: key inequality}.
The inequality~\eqref{eq: key inequality} with $p = 1$ trivially holds from Lemma~\ref{lem: f(t + theta)g(t)}.
Therefore, we only have to consider the case $p \in \rbra*{1, \infty}$ and $I \ne 0$.
Let
\begin{equation*}
	M(t) \coloneqq \rbra*{ \int_{-r}^{-t} \abs*{\phi\rbra*{t + \theta}} \mspace{2mu} \dd V_\eta(\theta) }^{p - 1}
\end{equation*}
for each $t \in \sbra*{0, r}$.
Since
\begin{equation*}
	\int_{-r}^{-t} \abs*{\phi\rbra*{t + \theta}} \mspace{2mu} \dd V_\eta(\theta)
	= \int_{-r}^0 \abs*{\Bar{\phi}\rbra*{t + \theta}} \mspace{2mu} \dd V_\eta(\theta) - \abs*{\phi(0)} \cdot \sbra*{V_\eta(0) - V_\eta\rbra*{-t}},
\end{equation*}
the left-hand side is a Riemann integrable function of $t \in \sbra*{0, r}$.
Therefore, the function $\sbra*{0, r} \ni t \mapsto M(t) \in \mathbb{R}$ is also Riemann integrable.
Then we have
\begin{align*}
	I^p
	&= \int_0^r \rbra*{ \int_{-r}^{-t} \abs*{\phi\rbra*{t + \theta}} \mspace{2mu} \dd V_\eta(\theta) } M(t) \mspace{2mu} \dd t \\
	&= \int_{-r}^0 \rbra*{ \int_0^{-\theta} \abs*{\phi\rbra*{t + \theta}}M(t) \mspace{2mu} \dd t } \mspace{2mu} \dd V_\eta(\theta)
\end{align*}
from Lemma~\ref{lem: f(t + theta)g(t)}.
By applying H\"{o}lder's inequality,
\begin{align*}
	\int_0^{-\theta} \abs*{\phi\rbra*{t + \theta}}M(t) \mspace{2mu} \dd t
	&\le \rbra*{ \int_0^{-\theta} \abs*{\phi\rbra*{t + \theta}}^p \mspace{2mu} \dd t }^{1/p} \cdot \rbra*{ \int_0^{-\theta} M(t)^q \mspace{2mu} \dd t }^{1/q} \\
	&\le \rbra*{ \int_0^{-\theta} \abs*{\phi\rbra*{t + \theta}}^p \mspace{2mu} \dd t }^{1/p} \cdot \rbra*{ \int_0^r M(t)^q \mspace{2mu} \dd t }^{1/q}
\end{align*}
hold, where $q \in \rbra*{1, \infty}$ is the exponent conjugate to $p$.
Since $q = p/\rbra*{p - 1}$, the constant in the right-hand side is calculated as
\begin{equation*}
	\rbra*{ \int_0^r M(t)^q \mspace{2mu} \dd t }^{1/q}
	= \sbra*{ \int_0^r \rbra*{ \int_{-r}^{-t} \abs*{\phi\rbra*{t + \theta}} \mspace{2mu} \dd V_\eta(\theta) }^p \mspace{2mu} \dd t }^{1/q}
	= I^{p/q}.
\end{equation*}
Thus, the above argument shows
\begin{equation*}
	I^p
	\le I^{p/q} \cdot \int_{-r}^0 \rbra*{ \int_0^{-\theta} \abs*{\phi\rbra*{t + \theta}}^p \mspace{2mu} \dd t }^{1/p} \mspace{2mu} \dd V_\eta(\theta),
\end{equation*}
which yields \eqref{eq: key inequality}.
Since the right-hand side of \eqref{eq: key inequality} is estimated by $\Var(\eta)\norm*{\phi}_{L^p\sbra*{-r, 0}}$, the inequality is obtained.
\end{proof}

\begin{remark}
Lemma~\ref{lem: L^p norm of g_phi} for the case of $p = 2$ has been discussed in \cite[Theorem~2.1]{Delfour--Manitius_1980} with a different argument of the proof and with a different definition of $g\rbra*{\argdot; \phi}$.
\end{remark}

\subsection{Proof of Theorem~\ref{thm: main result} for \texorpdfstring{$1 < p < \infty$}{1 < p < infty}}

In this subsection, we give the proof of Theorem~\ref{thm: main result} for $1 < p < \infty$.

\begin{proof}[Proof of Theorem~\ref{thm: main result} for $1 < p < \infty$]
Let $\phi \in \mathcal{M}^p$ be given.
By showing
\begin{equation*}
	\Dot{G}\rbra*{\argdot; \phi} \in \mathcal{L}^p_\mathrm{loc}\rbra*{[0, \infty), \mathbb{K}^n},
\end{equation*}
the statements of Theorem~\ref{thm: main result} for $1 < p < \infty$ are obtained with $f\rbra*{\argdot; \phi} = \Dot{G}\rbra*{\argdot; \phi}$ as the proof of Theorem~\ref{thm: DE satisfied by mild sol}.

\textbf{Step 1: Definition of $G\rbra*{\argdot; \psi}$ for $\psi \in L^p\rbra*{\sbra*{-r, 0}, \mathbb{K}^n}$.}
Since
\begin{equation*}
	G\rbra*{t; \phi}
	= \int_{-r}^0 \dd \eta(\theta) \mspace{2mu} \rbra*{\int_\theta^0 \phi(s) \mspace{2mu} \dd s} + \int_{-r}^{-t} \dd \eta(\theta) \mspace{2mu} \rbra*{\int_0^{t + \theta} \phi(s) \mspace{2mu} \dd s}
\end{equation*}
holds for all $t \in \sbra*{0, r}$ and $G\rbra*{\argdot; \phi}$ is constant on $[r, \infty)$, one can define $G\rbra*{\argdot; \psi} \colon [0, \infty) \to \mathbb{K}^n$ for $\psi \in L^p \coloneqq L^p\rbra*{\sbra*{-r, 0}, \mathbb{K}^n}$ in the same way.
Then the definition yields that
\begin{equation}\label{eq: estimate for G(argdot; psi)}
	\abs*{G\rbra*{t; \psi}}
	\le 2\Var(\eta)\norm*{\psi}_{L^1\sbra*{-r, 0}}
\end{equation}
holds for all $t \in [0, \infty)$ and all $\psi \in L^p$.

\textbf{Step 2: Extension of $g\rbra*{\argdot; \psi}$ for $\psi \in C$.}
From Lemma~\ref{lem: L^p norm of g_phi}, $\rbra*{ C, \norm*{\argdot}_{L^p\sbra*{-r, 0}} } \ni \psi \mapsto g_\psi \in L^p\rbra*{\sbra*{0, r}, \mathbb{K}^n}$ is a bounded linear operator.
Since the subset $C$ is dense in $L^p$, there exists a unique bounded linear operator $T \colon L^p \to L^p\rbra*{\sbra*{0, r}, \mathbb{K}^n}$ such that
\begin{equation*}
	T\psi = g_\psi
\end{equation*}
holds for all $\psi \in C$.

\textbf{Step 3: Derivative of $G\rbra*{\argdot; \psi}|_{\sbra*{0, r}}$ for $\psi \in L^p$.}
For each given $\psi \in L^p$, we choose a sequence $\rbra*{\psi_j}_{j = 1}^\infty$ in $C$ so that $\norm*{\psi - \psi_j}_{L^p\sbra*{-r, 0}} \to 0$ as $j \to \infty$.
Then for all $t \in \sbra*{0, r}$, we have
\begin{equation*}
	G\rbra*{t; \psi}
	= \lim_{j \to \infty} G\rbra*{t; \psi_j}
	= \lim_{j \to \infty} \int_0^t g_{\psi_j}(s) \mspace{2mu} \dd s
\end{equation*}
from \eqref{eq: estimate for G(argdot; psi)} and Theorem~\ref{thm: g(argdot; phi)}.
By combining this and $\norm*{T\psi}_{L^p\sbra*{0, r}} \le \norm*{T}\norm*{\psi}_{L^p\sbra*{-r, 0}}$ for all $\psi \in L^p$, 
\begin{equation*}
	G\rbra*{t; \psi}
	= \lim_{j \to \infty} \int_0^t \rbra*{T\psi_j}(s) \mspace{2mu} \dd s
	= \int_0^t \rbra*{T\psi}(s) \mspace{2mu} \dd s
\end{equation*}
is concluded.

\textbf{Step 4: Conclusion.}
Step 3 shows $\Dot{G}\rbra*{\argdot; \phi} \in \mathcal{L}^p_\mathrm{loc}\rbra*{[0, \infty), \mathbb{K}^n}$.
This completes the proof.
\end{proof}

\section{Discussion}\label{sec: discussion}

We compare Theorem~\ref{thm: main result} with results obtained by Delfour and Manitius~\cite{Delfour--Manitius_1980}.
In that paper, the authors interpret the linear RFDE~\eqref{eq: linear RFDE} under an initial condition $x_0 = \phi \in \mathcal{M}^2$ as a differential equation
\begin{equation}\label{eq: Delfour--Manitius}
	\Dot{x}(t) = \int_{-t}^0 \dd \eta(\theta) \mspace{2mu} x\rbra*{t + \theta} +
	\begin{cases}
		(H\phi)\rbra*{-t} & (\text{a.e.\ $t \in \sbra*{0, r}$}) \\
		0 & (t \in (r, \infty)).
	\end{cases}
\end{equation}
Here $H\phi$ is an $L^2$-function determined by $\phi$ with the density argument which appeared at the proof of Theorem~\ref{thm: main result} for $1 < p < \infty$.
See \cite[Subsection~2.1]{Delfour--Manitius_1980} for the detail.
Theorem~\ref{thm: main result} reveals a connection between the differential equation~\eqref{eq: Delfour--Manitius} used in \cite{Delfour--Manitius_1980} and the mild solutions introduced in \cite{Nishiguchi_2023}.

\appendix

\section{Theorems on Riemann--Stieltjes integrals}\label{sec: RS integral}

\subsection{A result on iterated Riemann--Stieltjes integrals}

The following is a result on iterated Riemann--Stieltjes integrals (e.g., see \cite[Theorem~15a]{Widder_1941}, \cite{Hewitt_1960}).
See also \cite[Theorem~3.8]{Nishiguchi_2023} for a related result.

\begin{theorem}\label{thm: Fubini thm for RS integrals of scalar-valued functions}
Let $\sbra*{a, b}$, $[c, d]$ be closed and bounded intervals of $\mathbb{R}$, and $\alpha \colon \sbra*{a, b} \to \mathbb{K}$, $\beta \colon [c, d] \to \mathbb{K}$ be functions of bounded variation.
If $f \colon \sbra*{a, b} \times [c, d] \to \mathbb{K}$ is a continuous function, then
\begin{equation}\label{eq: iterated RS integrals}
	\int_c^d \biggl( \int_a^b f(x, y) \mspace{2mu} \dd \alpha(x) \biggr) \mspace{2mu} \dd \beta(y)
	= \int_a^b \biggl( \int_c^d f(x, y) \mspace{2mu} \dd \beta(y) \biggr) \mspace{2mu} \dd \alpha(x)
\end{equation}
holds.
\end{theorem}

We note that the functions
\begin{align*}
	&[c, d] \ni y \mapsto \int_a^b f(x, y) \mspace{2mu} \dd \alpha(x) \in \mathbb{K}, \\
	&\sbra*{a, b} \ni x \mapsto \int_c^d f(x, y) \mspace{2mu} \dd \beta(y) \in \mathbb{K}
\end{align*}
are continuous by the uniform continuity of $f$.
Therefore, both the left-hand side and right-hand side of \eqref{eq: iterated RS integrals} make sense as Riemann--Stieltjes integrals.
The proof of Theorem~\ref{thm: Fubini thm for RS integrals of scalar-valued functions} mentioned in \cite{Hewitt_1960} relies on the Stone--Weierstrass theorem (e.g., see \cite[7.32 Theorem]{Rudin_1976}).

\subsection{A variant of Minkowski's integral inequality}

The following is a result on a variant of Minkowski's integral inequality.
See \cite[Exercise~16 of Chapter~8]{Rudin_1987} for a statement in the setting of measure theory.

\begin{theorem}\label{thm: Minkowski's integral ineq}
Let $\sbra*{a, b}$, $[c, d]$ be closed and bounded intervals of $\mathbb{R}$, and $\alpha \colon \sbra*{a, b} \to \mathbb{R}$, $\beta \colon [c, d] \to \mathbb{R}$ be monotonically increasing functions.
If $f \colon \sbra*{a, b} \times [c, d] \to \mathbb{R}$ is a continuous function, then for any $p \in \rbra*{1, \infty}$,
\begin{equation}\label{eq: Minkowski's integral ineq}
	\left( \int_c^d \abs*{\int_a^b f(x, y) \mspace{2mu} \dd \alpha(x)}^p \mspace{2mu} \dd \beta(y) \right)^{1/p}
	\le \int_a^b \biggl( \int_c^d \abs*{f(x, y)}^p \mspace{2mu} \dd \beta(y) \biggr)^{1/p} \mspace{2mu} \dd \alpha(x)
\end{equation}
holds.
\end{theorem}

We note that the statement in the above theorem also holds when $p = 1$.

\begin{proof}[A proof of Theorem~\ref{thm: Minkowski's integral ineq}]
Let $I$ be the left-hand side of \eqref{eq: Minkowski's integral ineq}.
When $I = 0$, the inequality~\eqref{eq: Minkowski's integral ineq} is trivial.
Therefore, we may assume $I \ne 0$.
Let
\begin{equation*}
	M(y) \coloneqq \abs*{\int_a^b f(x, y) \mspace{2mu} \dd \alpha(x)}^{p - 1}
\end{equation*}
for each $y \in [c, d]$.
Then we have
\begin{align*}
	I^p
	&= \int_c^d \abs*{\int_a^b f(x, y) \mspace{2mu} \dd \alpha(x)} M(y) \mspace{2mu} \dd \beta(y) \\
	&\le \int_c^d \biggl( \int_a^b \abs*{f(x, y)} \mspace{2mu} \dd \alpha(x) \biggr) M(y) \mspace{2mu} \dd \beta(y) \\
	&= \int_a^b \biggl( \int_c^d \abs*{f(x, y)}M(y) \mspace{2mu} \dd \beta(y) \biggr) \mspace{2mu} \dd \alpha(x)
\end{align*}
by applying Theorem~\ref{thm: Fubini thm for RS integrals of scalar-valued functions}, where the monotonically increasing property of $\alpha, \beta$ and the continuity of the function
\begin{equation*}
	\sbra*{a, b} \times [c, d] \ni (x, y) \mapsto \abs*{f(x, y)}M(y) \in \mathbb{R}
\end{equation*}
are used.
By applying H\"{o}lder's inequality for Riemann--Stieltjes integrals,
\begin{equation*}
	\int_c^d \abs*{f(x, y)}M(y) \mspace{2mu} \dd \beta(y)
	\le \biggl( \int_c^d \abs*{f(x, y)}^p \mspace{2mu} \dd \beta(y) \biggr)^{1/p} \cdot \biggl( \int_c^d M(y)^q \mspace{2mu} \dd \beta(y) \biggr)^{1/q}
\end{equation*}
holds, where $q \in \rbra*{1, \infty}$ is the exponent conjugate to $p$.
Since $q = p/(p - 1)$, the constant in the right-hand side is calculated as
\begin{equation*}
	\biggl( \int_c^d M(y)^q \mspace{2mu} \dd \beta(y) \biggr)^{1/q}
	= \biggl( \int_c^d \abs*{\int_a^b f(x, y) \mspace{2mu} \dd \alpha(x)}^p \mspace{2mu} \dd \beta(y) \biggr)^{1/q}
	= I^{p/q}.
\end{equation*}
Thus, the above argument shows
\begin{equation*}
	I^p
	\le I^{p/q} \int_a^b \biggl( \int_c^d \abs*{f(x, y)}^p \mspace{2mu} \dd \beta(y) \biggr)^{1/p} \mspace{2mu} \dd \alpha(x),
\end{equation*}
which yields \eqref{eq: Minkowski's integral ineq}.
\end{proof}

\section*{Acknowledgements}
\addcontentsline{toc}{section}{Acknowledgements}

This work was supported by JSPS Grant-in-Aid for Young Scientists Grant Number JP23K12994.

\section*{Conflict of Interest}

The author states that there is no conflict of interest.

\section*{Data Availability}
Data sharing not applicable to this article as no datasets were generated or analyzed during the current study.


\begin{thebibliography}{99}
\addcontentsline{toc}{section}{References}

\bibitem{Bernier--Manitius_1978} 
	C. Bernier and A. Manitius,
	\textit{On semigroups in $\mathbb{R}^n \times L^p$ corresponding to differential equations with delays},
	Canadian J. Math. \textbf{30} (1978), 897--914.
	\url{https://doi.org/10.4153/CJM-1978-078-6}.

\bibitem{Delfour_1980} 
	M. C. Delfour,
	\textit{The largest class of hereditary systems defining a $C_0$ semigroup on the product space},
	Canad. J. Math. \textbf{32} (1980), no.~4, 969--978.
	\url{https://doi.org/10.4153/CJM-1980-074-8}.

\bibitem{Delfour--Manitius_1980} 
	M. C. Delfour and A. Manitius,
	\textit{The structural operator F and its role in the theory of retarded systems. I},
	J. Math. Anal. Appl. \textbf{73} (1980), no.~2, 466--490.
	\url{https://doi.org/10.1016/0022-247X(80)90292-9}.

\bibitem{Delfour--Mitter_1972_hereditary} 
	M. C. Delfour and S. K. Mitter,
	\textit{Hereditary differential systems with constant delays. I. General case},
	J. Differential Equations \textbf{12} (1972), 213--235; erratum, ibid. \textbf{14} (1973), 397.
	\url{https://doi.org/10.1016/0022-0396(72)90030-7}.

\bibitem{Diekmann--Gyllenberg--Thieme_1993} 
	O. Diekmann, M. Gyllenberg and H. R. Thieme,
	\textit{Perturbing semigroups by solving Stieltjes renewal equations},
	Differential Integral Equations \textbf{6} (1993), no.~1, 155--181.
	\url{https://doi.org/10.57262/die/1371214985}.

\bibitem{Hale_1963_stability} 
	J. K. Hale,
	\textit{A stability theorem for functional-differential equations},
	Proc. Nat. Acad. Sci. U.S.A. \textbf{50} (1963), 942--946.
	\url{https://doi.org/10.1073/pnas.50.5.942}.

\bibitem{Hale_1977} 
	J. K. Hale,
	\textit{Theory of functional differential equations},
	Second edition. Appl. Math. Sci., Vol.~3. Springer-Verlag, New York, 1977.
	\url{https://doi.org/10.1007/978-1-4612-9892-2}.

\bibitem{Hale_2006_history} 
	J. K. Hale,
	\textit{History of delay equations},
	in: Delay differential equations and applications, 1--28, 
	NATO Sci. Ser. II Math. Phys. Chem., Vol.~205, Springer, Dordrecht, 2006.
	\url{https://doi.org/10.1007/1-4020-3647-7_1}.

\bibitem{Hale--VerduynLunel_1993} 
	J. K. Hale and S. M. Verduyn Lunel,
	\textit{Introduction to functional differential equations},
	Appl. Math. Sci., Vol.~99. Springer-Verlag, New York, 1993.
	\url{https://doi.org/10.1007/978-1-4612-4342-7}.

\bibitem{Hewitt_1960} 
	E. Hewitt,
	\textit{Integration by parts for Stieltjes integrals},
	Amer. Math. Monthly \textbf{67} (1960), 419--423.
	\url{https://doi.org/10.1080/00029890.1960.11989522}.

\bibitem{Krasovskii_1963} 
	N. N. Krasovskii,
	\textit{Stability of motion. Applications of Lyapunov's second method to differential systems and equations with delay},
	Translated by J. L. Brenner Stanford University Press, Stanford, Calif. 1963.

\bibitem{Nishiguchi_2023} 
	J. Nishiguchi,
	\textit{Mild solutions, variation of constants formula, and linearized stability for delay differential equations},
	Electron. J. Qual. Theory Differ. Equ. \textbf{2023}, No.~32, 1--77.
	\url{https://doi.org/10.14232/ejqtde.2023.1.32}.

\bibitem{Rudin_1976} 
	W. Rudin,
	\textit{Principles of mathematical analysis},
	Third edition. McGraw-Hill Book Co., New York-Auckland-D\"{u}sseldorf, 1976.

\bibitem{Rudin_1987} 
	W. Rudin,
	\textit{Real and complex analysis},
	Third edition. McGraw-Hill Book Co., New York, 1987.

\bibitem{Shapiro_2018} 
	J. H. Shapiro,
	\textit{Volterra adventures},
	Student Mathematical Library, Vol.~85. American Mathematical Society, Providence, RI, 2018.
	\url{https://doi.org/10.1090/stml/085}.

\bibitem{Widder_1941} 
	D. V. Widder,
	\textit{The Laplace transform},
	Princeton Mathematical Series, Vol.~6. Princeton University Press, Princeton, N. J., 1941.

\end{thebibliography}
\end{document}